\documentclass[11pt]{amsart}

\usepackage{wasysym,amsthm,framed,amssymb,amsfonts,amsmath,enumerate,graphicx,tikz,bbm}
\usetikzlibrary{matrix,arrows,decorations.pathmorphing}
\usepackage[margin=1in]{geometry}

\newcommand{\boldY}{{\bf{y}}}
\newcommand{\boldOne}{\mathbbm{1}}

\newcommand{\Z}{\mathbb{Z}}
\newcommand{\N}{\mathbb{N}}
\newcommand{\R}{\mathbb{R}}

\newcommand{\cone}{\mathrm{cone}}

\renewcommand{\phi}{\varphi}

\newcommand{\E}{\mathrm{E}}

\newcommand{\Tor}{\mathrm{Tor}}
\DeclareMathOperator*{\rank}{rank}
\DeclareMathOperator*{\LT}{LT}\DeclareMathOperator*{\LC}{LC}

\newcommand\commentout[1]{}

\newtheorem{theorem}{Theorem}[section]

\newtheorem{corollary}[theorem]{Corollary}

\newtheorem{lemma}[theorem]{Lemma}

\theoremstyle{remark}
\newtheorem{example}[theorem]{Example}
\newtheorem{remark}[theorem]{Remark}

\theoremstyle{definition}
\newtheorem{definition}[theorem]{Definition}
\newtheorem{construction}[theorem]{Construction}

\begin{document}

\title{Rationality of Poincar\'e series for a Family of Lattice Simplices}

\author{Benjamin Braun}
\address{Department of Mathematics\\
         University of Kentucky\\
         Lexington, KY 40506--0027}
\email{benjamin.braun@uky.edu}

\author{Brian Davis}
\address{}
\email{brian.davis@uky.edu}

\date{30 October 2020}

\thanks{The first author was partially supported by grant H98230-16-1-0045 from the U.S. National Security Agency.
This material is also based in part upon work supported by the National Science Foundation under Grant No. DMS-1440140 while the first author was in residence at the Mathematical Sciences Research Institute in Berkeley, California, during the Fall 2017 semester.}

\begin{abstract}
We investigate multi-graded Gorenstein semigroup algebras associated with an infinite family of reflexive lattice simplices.
For each of these algebras, we prove that their multigraded Poincar\'e series is rational.
Our method of proof is to produce for each algebra an explicit minimal free resolution of the ground field, in which the resolution reflects the recursive structure encoded in the denominator of the finely-graded Poincar\'e series.
These algebras are not Koszul and therefore rationality is non-trivial.
Our proof techniques are based on explicit combinatorial considerations, with an emphasis on understanding the combinatorial structure underlying the resolution.
Further, our results demonstrate how interactions between multivariate and univariate rational generating functions can create subtle complications when attempting to use rational Poincar\'e series to inform the construction of minimal resolutions.
\end{abstract}

\maketitle

%%%%%%%%%%%%%%%%%%%%%%%%%%%%%%%%%%%%%%%%%%%%%%%%%%%%%%%%%%%%%%%%%%%%%%%%%%
\section{Introduction}

\subsection{Background}
For a field $K$ of characteristic $0$ and $\Z^n$-graded quotient ring $R=K[x_1,\dots,x_n]/I$, the \emph{Betti number} $\beta_{i,\alpha}^R(K)$ is the rank of the $\alpha$-graded component of the $i$-th module in a minimal free resolution of $K\cong R/(x_1,\ldots,x_n)$ over $R$. 
The \emph{Poincar\'e series} of $K$ over $R$ is the generating function
\[P_K^{R}(z,{\bf{y}}):=\sum_{\alpha\in\Z^n}\sum_{i\geq0}\beta_{i,\alpha}^R(K)z^i\boldY^\alpha=\sum_{\alpha\in\Z^n}\sum_{i\geq0}\dim_K\Tor_{i,\alpha}^{R}(K,K)z^i\boldY^\alpha,\] 
where $\boldY^\alpha$ denotes $y_1^{\alpha_1}\cdots y_n^{\alpha_n}$.
Unlike resolutions of finitely generated graded modules over $K[x_1,\dots,x_n]$, these Betti numbers are not necessarily eventually zero; as a result, interesting questions arise as to their behavior.
Traditionally, the Poincar\'e series has been considered with respect to only the variable $z$, setting $y_i=1$ for all $i$.
In this context, a classical question of Serre-Kaplansky was whether or not the Poincar\'e series is rational for all such $R$. 
This question was answered in the negative by Anick~\cite{IrrPoincare}, and much subsequent work has focused on determining the properties of $R$ that lead to rationality or irrationality~\cite{McCulloughPeevaSurvey}.

When $I$ is generated by monomials, as in the case of Stanley-Reisner theory, the Poincar\'e series is known to be rational.
Berglund, Blasiak, and Hersh~\cite{CombinatoricsPoincareMonomial} describe a combinatorial method for computing the rational form. 
Less is known about infinite graded resolutions associated to quotients by another important class of ideals in combinatorics, toric ideals~\cite{PeevaToric}. 
An example of a toric ring with transcendental Poincar\'e series was found by Roos and Sturmfels~\cite{irrationalPoincare}, and it is known by work of Gasharov, Peeva, and Welker~\cite{GenericToricRational} that quotients arising from generic toric ideals have rational Poincar\'e series.
It is not known whether rationality or irrationality of the Poincar\'e series is the more ``common'' property for toric rings.
One property that implies rationality for Poincar\'e series is when $R$ is \emph{Koszul}, i.e., when $K$ admits a linear minimal free resolution over $R$.
Rationality follows from the fact that Koszul rings have Hilbert and Poincar\'e series satisfying the functional equation
\begin{equation}
\label{eqn:koszul} H_R(-z)P_K^R(z)=1 \, . 
\end{equation}
Much of what is known about rationality of the Poincar\'e series in the toric setting is a consequence of the rationality of Hilbert series and a proof of Koszuality of  the specific algebras under consideration.

Another line of investigation relevant to this paper is the rationality of Poincar\'e series for certain Gorenstein rings.
Elias and Valla proved~\cite{EliasValla} that the Poincar\'e series of an almost stretched Gorenstein local ring of dimension $d$ and embedding codimension $h$ is given by
\begin{equation}
\label{eqn:gorpoincare} \frac{(1+z)^d}{1-hz+z^2} \, . 
\end{equation}
Rossi and \c Sega proved~\cite{RossiSega} that for any finitely-generated module over a compressed Gorenstein local ring of socle degree $2\leq s\neq 3$, the Poincar\'e series is rational.
One aspect of these results on Gorenstein local rings is that the proofs do not directly construct an explicit minimal free resolution satisfying the recursive behavior encoded in the rational series, in the sense that the combinatorial structure of the underlying resolution is not made explicit.

\subsection{Our Contributions}

For a family of lattice polytopes described in Definition~\ref{def:simplex} and denoted by $\Delta_2^m$, we prove in Theorem~\ref{mainTheorem} that the Poincar\'e series for their associated semigroup algebra $\E(\Delta_2^m)$ (defined in the next section) using a fine grading is rational, with structure similar to~\eqref{eqn:gorpoincare}.
However, the recurrence given in~\eqref{eqn:gorpoincare} is realized only after specializing our fine grading to a coarse grading and then algebraically canceling.
Our method of proof is to inductively construct an explicit resolution of $K$ over the quotient of $\E(\Delta_2^m)$ by a linear system of parameters such that the resolution reflects the recursive structure encoded in the denominator of the finely-graded Poincar\'e series.
While there are various algebraic techniques available for establishing rationality of Poincar\'e series, from our perspective as combinatorialists it is valuable to develop and understand explicit combinatorial constructions of resolutions and their relationship to rationality.
Our proof techniques follow this philosophy.
We also show that $\E(\Delta_2^m)$ is not Koszul, and therefore rationality does not follow from~\eqref{eqn:koszul}.
We believe that the results in this paper will be of interest to both combinatorialists and commutative algebraists, for the following reasons.

\begin{itemize}
\item Given a lattice polytope $P$, there has been fruitful investigation of the Hilbert series of $\E(P)$, i.e., the Ehrhart series of $P$, in relation to the geometry and arithmetics of $P$.
We believe that a similar investigation should be conducted for Poincar\'e series.
Our work is a contribution in this direction.
\item For an arbitrary lattice simplex $P$, the arithmetic properties of the fundamental parallelepiped of $P$ should significantly impact the behavior of the Poincar\'e series for $\E(P)$.
This influence should be more subtle than the interpretation of the Hilbert $h$-vector of $\E(P)$, i.e., the Ehrhart $h^*$-vector of $P$, in terms of lattice points in the fundamental parallelepiped.
Our results show how this works in a special case.
From a combinatorial perspective, it is of particular interest to determine the relationship between the combinatorial structure of the fundamental parallelepiped points and the combinatorial structure of the resulting resolution.
\item Our results demonstrate how interactions between multivariate and univariate rational generating functions that are ``typical'' in enumerative combinatorics can create subtle complications when attempting to use univariate rational Poincar\'e series to inform the construction of minimal resolutions.
\end{itemize}

The remainder of our paper is structured as follows.
In Section~\ref{sec:fpa} we describe a family of lattice simplices and their associated semigroup algebras. 
We introduce the Fundamental Parallelepiped Algebra for lattice simplices and explain its connection to Poincar\'e series. 
In Section~\ref{sec:tree} we present a tree whose weighted rank generating function is equal to the Poincar\'e series of the Fundamental Parallelepiped Algebra, and whose structure is related to the rationality of that formal power series. 
In Section~\ref{sec:rational} we state and prove our main result, Theorem~\ref{mainTheorem}, which gives a rational expression for the fine graded Poincar\'e series of the Fundamental Parallelepiped Algebra of an infinite family of lattice simplices.

%%%%%%%%%%%%%%%%%%%%%%%%%%%%%%%%%%%%%%%%%%%%%%%%%%%%%%%%%%%%%%%%%%%%%%%%%%

\section{The Fundamental Parallelepiped Algebra and $\Delta_2^m$}\label{sec:fpa}
For a \emph{lattice polytope} $P\subset \R^d$, i.e., a polytope whose vertices $\{v_1,\ldots,v_n\}$ lie in the integer lattice $\Z^d$, we define the \emph{cone over} $P$ to be 
\begin{equation}\label{eqn:cone}
\cone(P):=\left\{\sum_{i=1}^{n}\,a_i\cdot (1,v_i)\;:\;a_i\in\R_{\geq0}\right\}\subset\R\times \R^{d} \, ,
\end{equation}
where $(1,v_i)$ is the embedding of vertex $v_i$ in $\R^{1+d}$.
The set $\Lambda:=\cone(P)\cap \Z^{1+d}$ forms a semigroup under addition.
It is known that there exists a unique minimal set of additive generators for $\Lambda$ called the \emph{Hilbert basis} of $\cone(P)$. 
We call the associated semi-group ring $\E(P):=K[\Lambda]=K[\cone(P)\cap \Z^{1+d}]$ the \emph{Ehrhart ring} of $P$.
It is well-known that $\E(P)$ is a quotient of a polynomial ring by a toric ideal.
To the lattice point $z\in \Lambda$ we associate the formal basis element $e_z\in K[\Lambda]$.
When the degree of the algebra element corresponding to the lattice point $(m_0,\ldots,m_d)\in \cone(P)\cap \Z^{1+d}$ is defined to be $m_0$, the resulting Hilbert series is referred to in combinatorics as the \emph{Ehrhart series} of $P$.
Further, it is known by a theorem of Hochster~\cite{Hochster} that $\E(P)$ is a Cohen-Macaulay integral domain.
The study of Ehrhart series for lattice polytopes is an active area in geometric combinatorics, with direct connections to Hilbert series of Cohen-Macaulay algebras.
While the study of Ehrhart series is well-established in combinatorics, the study of Poincar\'e series for $\E(P)$ has not to our knowledge been the subject of explicit investigation by combinatorialists.
When a lattice polytope has the property that $\E(P)$ is Koszul, then knowledge of the Ehrhart series is equivalent to that of the Poincar\'e series by~\eqref{eqn:koszul}.

\textit{\textbf{We assume throughout the remainder of this paper that $P$ is a lattice simplex,}} i.e., $P$ has $d+1$ vertices $\{v_1,\ldots,v_{d+1}\}\subset\Z^d$.
There is a natural decomposition of $\cone(P)$ obtained by tiling the cone with copies of the \emph{fundamental parallelepiped} of $P$, defined as follows:
\begin{equation}
\label{eqn:fppdef} \Pi:=\left\{\sum_{i=1}^{d+1} \, a_i\cdot (1,v_i)\; :\;0\leq a_i<1\right\} 
\end{equation}
Consequently, every element of $\Lambda$ has a unique representation as the sum of a lattice point in the fundamental parallelepiped and a non-negative integer combination of the primitive ray generators $(1,v_i)$. 
When $P$ is a simplex, the Hilbert basis consists of the ray generators $(1,v_i)$ and the set $\{h_1,\ldots,h_m\}$ of minimal (additive) generators of lattice points in $\Pi$. 
We associate a variable $V_i$ to each $(1,v_i)$ and a variable $x_i$ to each $h_i$. 
This defines a surjective degree map $\deg(\cdot)$ from the set of monomials of $K[V_1,\dots,V_{d+1},x_1,\dots,x_m]$ onto $\Lambda$ by 
\[
\deg\left(\prod V_i^{s_i}\cdot\prod x_j^{r_j}\right)=\sum s_i(1,v_i) + \sum r_jh_j \, .
\]
Extending $\deg(\cdot)$ $K$-linearly, the semi-group ring $\E(P)=K[\Lambda]$ then has a presentation \[0\longrightarrow I\longrightarrow K[V_1,\dots,V_{d+1},x_1,\dots,x_m]\longrightarrow \E(P)\longrightarrow 0,\] where the toric ideal $I$ is generated by all binomials 
\[
\mathbf{V}^{u_V}\mathbf{x}^{u_x}-\mathbf{V}^{w_V}\mathbf{x}^{w_x}
\]
such that $\deg\left(\mathbf{V}^{u_V}\mathbf{x}^{u_x})=\deg(\mathbf{V}^{w_V}\mathbf{x}^{w_x}\right)$. 

\begin{definition}
\label{def:fpa}
The \emph{Fundamental Parallelepiped Algebra (FPA)} of a simplex $P$ with toric ideal $I$ is the quotient algebra 
\begin{align*}
\widehat{R}:= & K[\Lambda]/(e_{(1,v_i)}:i=1,\ldots,d+1) \\
 \cong  & K[V_1,\dots,V_{d+1},x_1,\dots,x_m]/(I+(V_i:i=1,\ldots,d+1))\, .
\end{align*}
\end{definition}
$\widehat{R}$ is finite dimensional as a vector space, with $K$-basis $\{e_z\,:\,z\in\Z^{d+1}\cap\Pi\}$. 
Letting $\varphi$ be the projection $V_i\mapsto0$ from $K[V_1,\dots,V_{d+1},x_1,\dots,x_m]$ onto $K[x_1,\dots,x_m]$, we see that $\widehat{R}$ is isomorphic to $K[x_1,\dots,x_m]$ modulo the ideal $J:=\varphi(I)$. 
The ideal $J$ is not itself toric, as it has both binomial and monomial generators. 
The binomial generators ${\bf{x}}^u-{\bf{x}}^w$ in $J$ correspond to additive identities $\deg({\bf{x}}^u)=\deg({\bf{x}}^w)\in\Pi$. 
The monomial generators of $J$ lift to monomials of $K[V_1,\dots,V_{d+1},x_1,\dots,x_m]$ that are in the same equivalence class modulo $I$ as a monomial divisible by some $V_i$, and so correspond to elements of $\Lambda\backslash\Pi$. 

Our interest in the Fundamental Parallelepiped Algebra $\widehat{R}$ stems from the fact that the $V_i$'s form a linear system of parameters for  $K[V_1,\dots,V_{d+1},x_1,\dots,x_m]/I$.
By Proposition 3.3.5 of \cite{InfiniteFree}, we have that  
\begin{equation}\label{hatResult}
P_K^{\E(\Delta)}(z)=(1+z)^{d+1}\; P_K^{\widehat{R}}(z) \, .
\end{equation} 
Thus we may study the Poincar\'e series of the toric ring $\E(\Delta)$ by considering the ``simpler'' Artinian ring $\widehat{R}$.

\begin{example}\label{2^1Example}Computation in Macaulay2 gives that for the simplex $\Delta$ with vertices $(1,0)$, $(0,1)$, and $(-2,-3)$, $\cone(\Delta)$ has Hilbert basis (and associated variables) given by the columns below:
\[\bordermatrix{
    ~&V_1& V_2& V_3 & x_1&x_2&x_3&x_4\cr
  &1&1&1&1&1&1&1\cr
   &0&1&-2&0&0&-1&-1\cr
     &1&0&-3&0&-1&-1&-2
       }
\]
The associated toric ideal $I$ is given by \begin{align*}
  I=&(V_1x_2-x_1^2,V_2V_3-x_2x_4,V_2x_4-x_2^2,V_3x_2-x_4^2,V_1x_4-x_1x_3,\\&x_1x_4-x_2x_3,V_2x_3-x_1x_2,V_3x_1-x_3x_4,V_1V_3-x_3^2).
\end{align*} The Ehrhart ring $E(\Delta)$ is isomorphic to $K[V_1,V_2,V_3,x_1,x_2,x_3,x_4]/I$ and 
\[
\widehat{R} \cong  K[x_1,x_2,x_3,x_4]/(x_1^2,x_2x_4,x_2^2,x_3^2,x_4^2,x_1x_3,x_1x_4-x_2x_3,x_1x_2,x_3x_4) \, .
\]
\end{example}

The following family of simplices are the main objects under investigation in this work.

\begin{definition}
\label{def:simplex}
Let $\Delta_2^m$ be the $(m+1)$--simplex whose vertices are the standard basis vectors in $\R^{m+1}$ together with the point $(-2,\dots,-2,-2m-1)\in\R^{m+1}$. 
\end{definition}

The simplices $\Delta_2^m$ form a subfamily of the set of lattice simplices $\mathcal{Q}:=\{\Delta_{(1,\mathbf{q})}:\mathbf{q}\in \Z_{>0}^{d}\}$, where for a strictly positive integer vector $\mathbf{q}$ the simplex $\Delta_{(1,\mathbf{q})}$ is defined to be the convex hull of the standard basis vectors and $-\mathbf{q}$.
Simplices in $\mathcal{Q}$ have been the subject of extensive recent investigation~\cite{LaplacianSimplicesDigraphs,delta1QPaper,BraunDavisSolus,BraunHanely,BraunLiu,LiuSolusUnimodalityPositivity,SolusNumeralSystems}.
In particular, the simplices $\Delta_2^m$ were recently studied by Braun, Davis, and Solus~\cite{BraunDavisSolus} in the context of reflexive simplices having the integer decomposition property and also having a unimodal Ehrhart $h^*$-polynomial.

\begin{theorem}\label{thm:Rquotient}
Let $\widehat{R}$ denote the fundamental parallelepiped algebra for $\Delta_2^m$.
The following isomorphism holds for all $m\geq 1$:
\[
\widehat{R}\cong K[x_1,x_2,x_3,x_4]/(x_1^2,x_2^2,x_3^2,x_4^{m+1},x_1x_2,x_1x_3,x_2x_4^m,x_3x_4^m,x_2x_3-x_1x_4)
\] 
Further, the quotient algebra has a $K$-vector space basis given by the cosets represented by the elements of
\[
\{1,x_1,x_4^{\ell+1},x_1x_4^{\ell+1},x_2x_4^{\ell},x_3x_4^{\ell}\}_{0\leq\ell\leq m-1} \, .
\] 
\end{theorem}

\begin{proof}
We first describe the fundamental parallelepiped $\Pi$ for $\Delta_2^m$ and identify additive relations among the generators of the lattice points in it.
As shown in~\cite{BraunDavisSolus},  lattice points in $\Pi$ are parameterized by integers $b$ in $[0,4m+1]$, with each $b$ corresponding to the lattice point
\[
z_b:=\begin{bmatrix} b-m\lfloor \frac{b}{2m+1}\rfloor-\lfloor b/2\rfloor\\-\lfloor \frac{b}{2m+1}\rfloor\\\vdots\\-\lfloor \frac{b}{2m+1}\rfloor\\-\lfloor b/2\rfloor\\\vdots\\-\lfloor b/2\rfloor\end{bmatrix}\in\Pi \, .
\]  

Considering the cases $b<2m+1$ and $b\geq 2m+1$, and then considering the parity of $b$, we see that for each choice $1\leq h\leq m$ of zeroth coordinate, we get exactly four solutions (presented as column vectors below):
\[\bordermatrix{
    ~&z_{2h-1}& z_{2h}       & z_{2(m+h)-1} & z_{2(m+h)}\cr
  &h&h&h&h\cr
   &0&0&-1&-1\cr
     &\vdots&\vdots&\vdots&\vdots\cr
   &0&0&-1&-1\cr
    &-h+1&-h&-(m+h)+1&-(m+h)\cr
     &\vdots&\vdots&\vdots&\vdots\cr
     &-h+1&-h&-(m+h)+1&-(m+h)
       }.\]

By Theorem 4.1 of \cite{BraunDavisSolus}, the simplex $\Delta_2^m$ has the integer decomposition property, implying that $\widehat{R}$ is generated by elements  with zeroth coordinate equal to 1, i.e., $e_{z_1}$, $e_{z_2}$, $e_{z_{2m+1}}$, and $e_{z_{2m+2}}$. 
Thus we may assume without loss of generality that additive identities in the fundamental parallelepiped have the form $z_b+z_{b'}=z_c+z_{c'}$, where $z_b$ and $z_c$ have zeroth coordinate equal to 1 and the zeroth coordinate of $z_{b'}$ and $z_{c'}$ is $h$. 
It follows by inspection that every such identity is of the form $z_1+z_{2(m+h)}=z_2+z_{2(m+h)-1}$ for some $h$ between 2 and $m$. Every such identity  may be written as 
\[
(h-1)z_2+(z_1+z_{2m+2})=(h-1)z_2+(z_2+z_{2m+1})\, ,
\] 
 and is therefore a consequence of the primitive additive identity 
\[
z_1+z_{2m+2}=z_2+z_{2m+1}\, .
\]

Now that we have a better understanding of the structure of $\widehat{R}$, we get close to our desired isomorphism by constructing the map
\[
\psi:K[x_1,x_2,x_3,x_4]/(x_1x_4-x_2x_3)\to\widehat{R}= K[\Lambda]/\left(e_{(1,v_i)}\right)
\]
 defined by algebraically extending the map on variables given by 
\[
x_1\mapsto e_{z_{2m+1}},\quad x_2\mapsto e_{z_{2m+2}},\quad x_3\mapsto e_{z_1},\quad x_4\mapsto e_{z_2} \, .
\]
To verify that $\psi$ is well-defined, consider a pair of monomials $\prod_i x_i^{s_i}$ and $\prod_j x_j^{t_j}$ that are in the same equivalence class. 
Then $(\prod_i x_i^{s_i})-(\prod_j x_j^{t_j})$ is in the ideal $(x_1x_4-x_2x_3)$, so that $(\prod_i x_i^{s_i})-(\prod_j x_j^{t_j})=t(x_1x_4-x_2x_3)$ for some $t$. 
It follows that $\psi(\prod_i x_i^{s_i})-\psi(\prod_j x_j^{t_j})=\psi(t)(e_{z_{2m+1}+z_2}-e_{z_{2m+2}+z_1})$ is zero, since, as we have seen,  $z_1+z_{2m+2}=z_2+z_{2m+1}$.
 It is straightforward to verify that the homomorphism $\psi$ is surjective.

We next determine the kernel of $\psi$.
Observe that since $2z_1$ is not among $z_3$, $z_4$, $z_{2m+3}$, and $z_{2m+4}$, we can conclude that $2z_1$ is not in $\Pi$. 
We can similarly conclude that $2z_{2m+1}$, $2z_{2m+2}$, $z_1+z_{2m+1}$, and $z_{2m+1}+z_{2m+2}$ are not in $\Pi$. 
We additionally see that $z_{4m+1}=mz_2+z_{2m+1}$, so that $mz_2+z_{1}$, $mz_2+z_{2m+2}$, and $(m+1)z_2$ are not in $\Pi$, since $\Pi$ contains a unique element with zeroth coordinate equal to $m+1$. 
Since $z_1+z_1$ is an element of $\Lambda$ but not $\Pi$, we conclude that $z_1+z_1=v_i+z$ for some $z$ in $\Lambda$. 
Thus $e_{z_1+z_1}=e_{z_1}^2=0$ in $\widehat{R}$. 
Similarly \[e_{z_1}^2=e_{z_1}e_{z_{2m+1}}=e_{z_{2m+1}}^2=e_{z_{2m+1}}e_{z_{2m+2}}=e_{z_{2m+2}}^2=e_{z_1}e_{z_2}^m=e_{z_{2m+2}}e_{z_2}^m=e_{z_2}^{m+1}=0\] in $\widehat{R}$. 
Thus the kernel of $\psi$ contains $(x_3^2,x_1x_3,x_1^2,x_1x_2,x_2^2,x_3x_4^m,x_2x_4^m,x_4^{m+1})$.

Finally, we count equivalence classes of monomials in the ring 
\[
K[x_1,x_2,x_3,x_4]/(x_1^2,x_2^2,x_3^2,x_4^{m+1},x_1x_2,x_1x_3,x_2x_4^m,x_3x_4^m,x_2x_3-x_1x_4) \, .
\] 
We only need to consider the monomials 1 and variables multiplied by powers of $x_4$, since $x_ix_j$ is either zero or equal to $x_4r_k$ for some $k$. 
It follows that it is a $(4m+2)$--dimensional $K$-vector space with basis 
\[
\{1,x_1,x_4^{\ell+1},x_1x_4^{\ell+1},x_2x_4^{\ell},x_3x_4^{\ell}\}_{0\leq\ell\leq m-1}
,
\] 
and with a surjective ring homomorphism $\hat{\psi}$ to the $(4m+2)$--dimensional vector space $\widehat{R}$, i.e., $\hat{\psi}$ is a ring isomorphism from 
\[
K[x_1,x_2,x_3,x_4]/(x_1^2,x_2^2,x_3^2,x_4^{m+1},x_1x_2,x_1x_3,x_2x_4^m,x_3x_4^m,x_2x_3-x_1x_4)
\] 
to the Fundamental Parallelepiped Algebra $\widehat{R}$. 
\end{proof}

As Koszul algebras must be defined by quadratically-generated ideals, the following corollary is immediate.

\begin{corollary}
For $m\geq 2$, $\widehat{R}$ is not Koszul for $\Delta_2^m$.
\end{corollary}

%%%%%%%%%%%%%%%%%%%%%%%%%%%%%%%%%%%%%%%%%%%%%%%%%%%%%%%%%%%%%%%%%%%%%%%%%%
\section{A Weighted Tree Encoding Betti Numbers}\label{sec:tree}
Our goal in this section is to define for $\Delta_2^m$ a $\Lambda$--weighted tree $T$ whose weighted rank generating function 
\[
T(z,\boldY):=\sum_{\epsilon\in T}z^{\rank(\epsilon)}\boldY^{\deg(\epsilon)}
\]
is equal to the $(\Lambda\times\N)$--graded Poincar\'e series 
\[
P_K^{\widehat{R}}(z,{\bf{y}}):=\sum_{\alpha\in\Lambda}\sum_{i\geq0}\dim_K\Tor_{i,\alpha}^{\widehat{R}}(K,K)z^i\boldY^\alpha \, ,
\] 
where $\boldY^\alpha$ means the multinomial $y_0^{\alpha_0}\cdots y_n^{\alpha_n}$.
To construct our weighted tree, we require the following general construction.

\begin{definition}\label{def:LT}
Let $P$ be a lattice simplex with fundamental parallelepiped algebra $\widehat{R}$ described as a quotient of a polynomial ring with a monomial term order $\prec_{\widehat{R}}$.
Assume that we have a \emph{distinguished monomial basis} for $\widehat{R}$ consisting of all monomials outside the $\prec_{\widehat{R}}$-leading term ideal for the defining ideal of $\widehat{R}$.
Let $d$ be a map between free finitely generated $\Lambda$-graded $\widehat{R}$-modules $M$ and $N$, where there is an ordering $\prec$ on the generators of $N$. 
Consider a generator $\epsilon$ of $M$, and let $\delta$ be the {$\prec$-minimal} support of $d(\epsilon)$ and $s$ the $\prec_{\widehat{R}}$-maximal monomial of $d(\epsilon)$ supported on $\delta$. 
If $\delta s$ is distinct for each $\epsilon$, then we say that $M$ \emph{can be ordered with respect to $d$}.
If $M$ can be ordered with respect to $d$, we define an ordering of the generators of $M$ as follows: $\epsilon\prec\epsilon'$ if $\delta\prec\delta'$ or if $\delta=\delta'$ and $s'\prec_{\widehat{R}} s$.
In this case, we define the \emph{leading term} map $\LT(\cdot)$ on the graded components of $M$ which projects each element onto the summand generated by its $\prec$-minimal support. 
For notational convenience, we define the \emph{leading coefficient} $\LC(\cdot)$ of an element to be the $\prec_{\widehat{R}}$-maximal monomial of its leading term. 
\end{definition}

For a given complex $(F,d)$ we denote by $F_{\leq n}$ the truncated complex 
\[
F_{\leq n}\;:\quad F_0\xleftarrow{d_1} F_1\xleftarrow{d_2} \cdots\xleftarrow{d_n} F_n \, .
\]
Observe that for a $\Lambda$-graded complex $F$ of free finitely generated $\widehat{R}$-modules, if $\LT(\cdot)$ is defined for the truncated complex $F_{\leq n}$ and the leading terms of $d_{n+1}(\epsilon)$ for generators $\epsilon$ of $F_{n+1}$ are all distinct, then $F_{n+1}$ can be ordered with respect to $d_{n+1}$.
In this case, we may define $\LT(\cdot)$ on $F_{n+1}$. 

\begin{definition}\label{def:Rorder}
For $\widehat{R}$ corresponding to $\Delta_2^m$ as given in Theorem~\ref{thm:Rquotient}, specify the ordering $\prec_{\widehat{R}}$ of the monomial $K$-basis by using the lexicographic order induced by the ordering
\[
1\prec_{\widehat{R}} x_1\prec_{\widehat{R}} x_2\prec_{\widehat{R}} x_3\prec_{\widehat{R}} x_4 \, ,
\] 
i.e. on our basis elements we have $x_ix_4^j\prec_{\widehat{R}} x_kx_4^\ell$ if $j<\ell$ or  $j=\ell$ and $x_i\prec_{\widehat{R}} x_k$. 
\end{definition}

\begin{example}
Let our simplex be $\Delta_2^m$ with $\widehat{R}$ as given in Theorem~\ref{thm:Rquotient}.
Consider the complex $F_{\leq2}$ below:
\[F_{\leq2}\;:\quad
\widehat{R}
 \xleftarrow[d_1]{\qquad}
    \widehat{R}^4 
    \xleftarrow[d_2]{\qquad}
    \widehat{R}^{15} \]  
where the map $d_1= \left( \begin{array}{cccc}
        x_1&x_2&x_3&x_4
    \end{array}\right)$ sends each $\delta_i\to x_i$
and $d_2$ is given by the matrix
\[
\bordermatrix{
~&\epsilon_1&\epsilon_2&\epsilon_3&\epsilon_4&\epsilon_5&\epsilon_6&\epsilon_7&\epsilon_8&\epsilon_9&\epsilon_{10}&\epsilon_{11}&\epsilon_{12}&\epsilon_{13}&\epsilon_{14}&\epsilon_{15}\cr 
\delta_1&x_1&x_2&x_3&x_4& 0&0&0&0&        0&0&0&0&0&0&0\cr
\delta_2&0&0&0&0&        x_1&x_2&x_3&x_4& 0&0&0&0&0&0&0\cr
\delta_3&0&0&0&0&        0&0&0&0&         x_1&x_2&x_3&x_4&0&0&0\cr
\delta_4&0&0&0&-x_1&     0&0&-x_1&-x_2 &     0&-x_1&0&-x_3&x_2x_4^{m-1}&x_3x_4^{m-1}&x_4^m} \, .
\] 
We see that $F_1$ can be ordered with respect to $d_1$, with the result that $\delta_1\prec \delta_2\prec \delta_3\prec \delta_4$.
Further, it is straightforward to verify that $F_2$ can be ordered with respect to $d_2$, and hence the leading term of the element $d_2(\epsilon_4)$ of $F_1$ is $\LT(d_2(\epsilon_4))=x_4\delta_1$ and the leading coefficient is $\LC(d_2(\epsilon_4))=x_4$.
\label{complexExample}
\end{example}

\begin{construction}\label{con:Rtree}
As in Definition~\ref{def:LT}, assume $P$ is a lattice simplex with fundamental parallelepiped algebra $\widehat{R}$ described as a quotient of a polynomial ring with a monomial term order $\prec_{\widehat{R}}$, together with a distinguished monomial basis.
Assume that $F$ is a resolution of a module $M$ over $\widehat{R}$ such that $F_n$ can be ordered with respect to $d_n$ and the order on $F_n$ is defined in this manner, with associated maps $\LT$ and $\LC$.
Construct a $\Lambda$-weighted tree $T$ whose elements are the generators of the summands of $F$, and whose cover relations are given by $\epsilon \gtrdot \delta$ if \,$\LT(d(\epsilon))=s\delta$. 
This also defines a labeling $\eta$ of the cover relations of $T$ where $\eta(\epsilon,\delta):=\LC(d(\epsilon))=s\in\widehat{R}$ (by construction a monomial).
\end{construction}
Note that if $F_0$ is cyclic, then $T$ is ranked, with the rank of an element equal to the graph distance between an element and the root of the tree in the Hasse diagram. 
For each element $\epsilon$ in $T$, there is a unique path $\hat{0}=t_0\lessdot t_1\lessdot\cdots\lessdot t_\ell=\epsilon$, where $\hat{0}$ is the generator of $F_0$. 
We define the degree of $\epsilon$ in $T$ to be 
\begin{equation}\label{degree}
\deg(\epsilon)=\sum_{i=0}^{\ell-1} \deg(\eta(t_i,t_{i+1}))\in\Lambda \, .
\end{equation}
This definition agrees with the internal degree of the summand generated by $\epsilon$, and the length $\ell$ of the chain from $\hat{0}$ to $\epsilon$ is precisely the homological degree where the summand sits. 
Thus there is a degree preserving bijection between summands of the complex $F$ and elements of $T$, so that $\displaystyle T(z,\boldY)=\sum_{\epsilon\in T}z^{\rank(\epsilon)}\boldY^{\deg(\epsilon)}$ is equal to 
\[
F(z,\boldY)=\sum_{k\geq0}\sum_{\alpha\in\Lambda}\rank \left[F_k\right]_\alpha z^k\boldY^\alpha \, .
\]
\begin{example} 
For the complex $F_{\leq2}$ of Example~\ref{complexExample}, the tree $T_{\leq2}$ is depicted in Figure~\ref{fig:tree}. 
Example~\ref{complexExample} implies that the cover label $\eta(\delta_1,\gamma)$ is equal to $x_1$ and the cover label $\eta(\epsilon_2,\delta_1)$ is equal to $x_2$, thus $\deg(\epsilon_2)$ is equal to $\deg(\eta(\delta_1,\gamma))+\deg(\eta(\epsilon_2,\delta_1))=\deg(x_1)+\deg(x_2)$.
After making a similar argument for each basis element in $F_{\leq 2}$, it follows that the generating function $T_{\leq2}(z,\boldY)$ is given by 
\[
T_{\leq2}(z,\boldY)=1+z(\boldY^{\deg(x_1)}+\boldY^{\deg(x_2)}+\boldY^{\deg(x_3)}+\boldY^{\deg(x_4)})+z^2(\boldY^{\deg(x_1)+\deg(x_1)}+\cdots+\boldY^{\deg(x_4)+\deg(x_4)}) \, .
\]

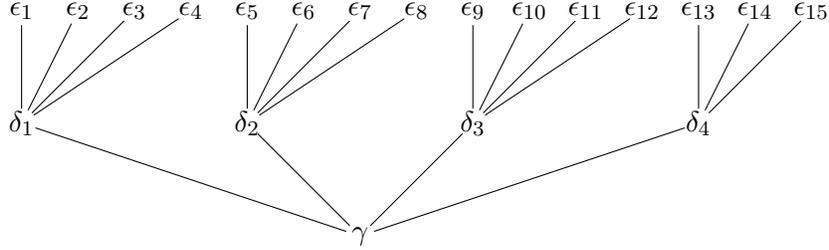
\begin{figure}[ht]
\begin{tikzpicture}[scale=.75]
\draw (6,-2) -- (0,0)--(0,2); 
\draw (1,2) -- (0,0)--(2,2); 
\draw (6,-2) -- (4,0)--(4,2); 
\draw (5,2) -- (4,0)--(6,2); 
\draw (6,-2) -- (8,0)--(8,2); 
\draw (9,2) -- (8,0)--(10,2); 
\draw (6,-2) -- (12,0)--(12,2); 
\draw (13,2) -- (12,0)--(14,2);
\draw (3,2) -- (0,0); 
\draw (7,2) -- (4,0); 
\draw (11,2) -- (8,0); 
\draw[white,fill=white] (6,-2) circle [radius=0.25];
\node at (6,-2) {$\gamma$};

\draw[white,fill=white] (4,0) circle [radius=0.25];
\node at (4,0) {$\delta_2$};
\draw[white,fill=white] (8,0) circle [radius=0.25];
\node at (8,0) {$\delta_3$};
\draw[white,fill=white] (12,0) circle [radius=0.25];
\node at (12,0) {$\delta_4$};
\draw[white,fill=white] (0,0) circle [radius=0.25];
\node at (0,0) {$\delta_1$};

\draw[white,fill=white] (0,2) circle [radius=0.25];
\node at (0,2) {$\epsilon_1$};
\draw[white,fill=white] (1,2) circle [radius=0.25];
\node at (1,2) {$\epsilon_2$};
\draw[white,fill=white] (2,2) circle [radius=0.25];
\node at (2,2) {$\epsilon_3$};
\draw[white,fill=white] (3,2) circle [radius=0.25];
\node at (3,2) {$\epsilon_4$};
\draw[white,fill=white] (4,2) circle [radius=0.25];
\node at (4,2) {$\epsilon_5$};
\draw[white,fill=white] (5,2) circle [radius=0.25];
\node at (5,2) {$\epsilon_6$};
\draw[white,fill=white] (6,2) circle [radius=0.25];
\node at (6,2) {$\epsilon_7$};
\draw[white,fill=white] (7,2) circle [radius=0.25];
\node at (7,2) {$\epsilon_8$};
\draw[white,fill=white] (8,2) circle [radius=0.25];
\node at (8,2) {$\epsilon_9$};
\draw[white,fill=white] (9,2) circle [radius=0.25];
\node at (9,2) {$\epsilon_{10}$};
\draw[white,fill=white] (10,2) circle [radius=0.25];
\node at (10,2) {$\epsilon_{11}$};
\draw[white,fill=white] (11,2) circle [radius=0.25];
\node at (11,2) {$\epsilon_{12}$};
\draw[white,fill=white] (12,2) circle [radius=0.25];
\node at (12,2) {$\epsilon_{13}$};
\draw[white,fill=white] (13,2) circle [radius=0.25];
\node at (13,2) {$\epsilon_{14}$};
\draw[white,fill=white] (14,2) circle [radius=0.25];
\node at (14,2) {$\epsilon_{15}$};
\end{tikzpicture}
\caption{The tree $T_{\leq2}$ for Example~\ref{complexExample}}
\label{fig:tree}
\end{figure}
\end{example}

The following lemma gives a sufficient condition for $F(z,\boldY)$ to be a rational function.

\begin{lemma}\label{matrixLemma}
Assume the setting of Construction~\ref{con:Rtree} and let $\{\lambda_i\}_{i\in[n]}$ denote the subset of elements of the distinguished monomial basis of $\widehat{R}$ that appear as labels in $T$.
Let the associated $\eta$-labeled tree $T$ have the property that the multiset $\left\{\eta(\epsilon,\delta):\,\epsilon\gtrdot\delta\right\}$ depends only on $\LC(d(\delta))$, i.e. for $\delta$ with $\LC(d(\delta))=\lambda_j$, there exists exactly $a_{i,j}$ elements $\epsilon$ in $T$ with $\eta(\epsilon,\delta)=\lambda_i$ (note that by hypothesis $a_{i,j}$ is either zero or one). 
Let $A$ be the $n\times n$ matrix with entries $A_{i,j}=a_{i,j}z\boldY^{\deg(\lambda_i)}$.
Then the generating function
\[
F(z,\boldY)=\sum_{k\geq0}\sum_{\alpha\in\Lambda}\rank \left[F_k\right]_\alpha z^k\boldY^\alpha
\] 
has a rational representation of the form 
\[
F(z,\boldY)=\frac{f(z,\boldY)}{\chi(z,\boldY,1)} \, ,
\] 
where $\chi(z,\boldY,t)$ is the characteristic polynomial of the matrix $A$ and the $z$-degree of $f(z,\boldY)$ is at most that of $\chi(z,\boldY,1)$. 
\end{lemma}

\begin{proof}
Let $b_{k,\alpha}^i$ be the number of rank $k$ elements $\epsilon$ of $T$ having degree $\alpha$ and with $\LC(d(\epsilon))=\lambda_i$, so that 
\[
\sum_{i=1}^n b_{k,\alpha}^i=\rank[F_k]_\alpha.
\] 
Define $B$ to be the $n\times1$ matrix whose $i$-th entry is $b_{1,\deg(\lambda_i)}^iz\boldY^{\deg(\lambda_i)}$.
Note that $b_{1,\deg(\lambda_i)}^i$ is equal to $1$ if $\lambda_i$ is equal to a single variable in $\widehat{R}$, and is equal to $0$ otherwise.

We prove by induction the claim that the matrix $\displaystyle A^{k}B$ is given by
\[
\left(A^{k}B\right)_i=\sum_{\alpha\in\Lambda}b_{k+1,\alpha}^iz^{k+1}\boldY^\alpha.
\] 
The base case $k=0$ is trivial. 
Assume the induction hypothesis and write $\displaystyle \left(A^{k}B\right)_i$ as 
\begin{align*}
\left(A^{k}B\right)_i&=\sum_{j=1}^nA_{i,j}\left(A^{k-1}B\right)_j\\
&=\sum_{j=1}^nA_{i,j}\left(\sum_{\alpha\in\Lambda}b_{k,\alpha}^jz^k\boldY^\alpha\right)\\
&=\sum_{\alpha\in\Lambda}\sum_{j=1}^nA_{i,j}b_{k,\alpha}^jz^k\boldY^\alpha\\
&=\sum_{\alpha\in\Lambda}\sum_{j=1}^na_{i,j}b_{k,\alpha}^jz^{k+1}\boldY^{\alpha+\deg(\lambda_i)}.
\end{align*} 
Observe that the coefficient of $z^{k+1}\boldY^\mu$ in the last line above is equal to $\sum_{j=1}^na_{i,j}b_{k,\mu-\deg(\lambda_i)}^j$, and that, by equation (\ref{degree}), the product $a_{i,j}b_{k,\mu-\deg(\lambda_i)}^j$ is precisely the number of elements $\epsilon$ of $T$ of degree $\mu$ and rank $k+1$ such that $\deg(\LC(d(\epsilon)))=\deg(\lambda_i)$. 
Thus 
\[
\left(A^{k}B\right)_i=\sum_{\mu\in\Lambda}b_{k+1,\mu}^iz^{k+1}\boldY^\mu,
\] completing the proof of the claim.

Defining $\boldOne_n$ to be the $1\times n$ matrix of 1's, note that \[\boldOne_n\cdot A^\ell\cdot B=\sum_{\alpha\in\Lambda}\rank[F_{\ell+1}]_\alpha \;z^{\ell+1}\boldY^\alpha.\]

Let $\chi\in K[z,\boldY,t]$ be the characteristic polynomial of $A$, so that $\chi(z,\boldY,A)=0$, and let \[\chi=t^d+\sum_{i=0}^{d-1}\chi_{i}t^{i},\] with $\displaystyle \chi_i=\sum_jq_{i,j}z^{r_{i,j}}\boldY^{s_{i,j}}$, where $q_{i,j}\in K$.  Then $\displaystyle A^d+\sum_{i=0}^{d-1}\sum_jq_{i,j}z^{r_{i,j}}\boldY^{s_{i,j}}A^{i}$ is the zero matrix. Left multiplying by $\boldOne_n A^{k-1}$ and right multiplying by $B$ yields \[\boldOne_nA^{d+k-1}B+\sum_{i=0}^{d-1}\sum_jq_{i,j}z^{r_{i,j}}\boldY^{s_{i,j}}\boldOne_n\cdot A^{i+k-1}\cdot B=0.\] 
Thus for all $k\geq1$,
\[
\sum_{\alpha\in\Lambda}\rank[F_{d+k}]_\alpha z^{d+k}\boldY^{\alpha}-\sum_{\mu\in\Lambda}\sum_{i=0}^{d-1}\sum_jq_{i,j}\rank[F_{i+k}]_{\mu}z^{i+k+r_{i,j}}\boldY^{\mu+s_{i,j}}=0.
\]
In particular, the coefficient of $z^{d+k}\boldY^\alpha$ on the left hand side is zero, so that 
\[
\rank[F_{d+k}]_\alpha+\sum_{i=0}^{d-1}\sum_jq_{i,j}\rank[F_{d+k-r_{i,j}}]_{\alpha-s_{i,j}}=0.
\] Since this is also the coefficient of $z^{d+k}\boldY^\alpha$ in the product $\displaystyle \chi(z,\boldY,1)\cdot F(z,\boldY)$, we see that this product is a power series in $K[[z,\boldY]]$ which vanishes in $z$-degree greater than $d$. Since $F_\ell$ is a finite direct sum for each $\ell$, the product is in fact a polynomial in $K[z,\boldY]$ and the result follows.
\end{proof}

%%%%%%%%%%%%%%%%%%%%%%%%%%%%%%%%%%%%%%%%%%%%%%%%%%%%%%%%%%%%%%%
\section{Rationality and $\Delta_2^m$} \label{sec:rational} 

In this section we prove the following theorem, our main result in this work.

\begin{theorem}\label{mainTheorem} 
For the simplex $\Delta_2^m$ with $\widehat{R}$ as given by Theorem~\ref{thm:Rquotient}, the $\widehat{R}$--module $K\cong \widehat{R}/(x_1,x_2,x_3,x_4)\widehat{R}$ has a minimal free resolution $F$ satisfying the hypotheses of Lemma~\ref{matrixLemma}.
The matrix $A$ resulting from Lemma~\ref{matrixLemma} in this case is given by
\begin{equation}\label{eqn:Amatrix}
\bordermatrix{
  ~&  x_1       & x_2 & x_3 & x_4 & x_2x_4^{m-1} & x_3x_4^{m-1}& x_4^m\cr
  x_1&z\boldY^{\deg(x_1)}       & z\boldY^{\deg(x_1)} & z\boldY^{\deg(x_1)} & 0 & z\boldY^{\deg(x_1)} & z\boldY^{\deg(x_1)}& 0\cr
      x_2&z\boldY^{\deg(x_2)}       & z\boldY^{\deg(x_2)} & z\boldY^{\deg(x_2)} & 0 & z\boldY^{\deg(x_2)} & z\boldY^{\deg(x_2)} & z\boldY^{\deg(x_2)} \cr  
      x_3&z\boldY^{\deg(x_3)}      & z\boldY^{\deg(x_3)} & z\boldY^{\deg(x_3)} & 0 & z\boldY^{\deg(x_3)} & z\boldY^{\deg(x_3)} & z\boldY^{\deg(x_3)} \cr
    x_4&z\boldY^{\deg(x_4)}      & z\boldY^{\deg(x_4)} & z\boldY^{\deg(x_4)}  & 0 & z\boldY^{\deg(x_4)} & z\boldY^{\deg(x_4)} & z\boldY^{\deg(x_4)} \cr
          x_2x_4^{m-1}&0       & 0 & 0 & z\boldY^{\deg(x_2x_4^{m-1})} & 0 & 0& 0\cr 
          x_3x_4^{m-1}&0       & 0 & 0 & z\boldY^{\deg(x_3x_4^{m-1})} & 0 & 0& 0\cr 
          x_4^m&0       & 0 & 0 & z\boldY^{\deg(x_4^{m})} & 0 & 0& 0\cr 
}
\end{equation} 

and thus the $(\Lambda\times\N)$-graded Poincar\'e series $P_K^{\widehat{R}}(z,\boldY)$ is given by
\[
\frac{1+z\boldY^{\deg(x_4)}}{1-z(\boldY^{\deg(x_1)}+\boldY^{\deg(x_2)}+\boldY^{\deg(x_3)})-z^2(\boldY^{\deg(x_2x_4^2)}+\boldY^{\deg(x_3x_4^2)}+\boldY^{\deg(x_4^3)})+z^3\boldY^{\deg(x_1x_4^3)}} \, .
\]
\end{theorem}

\begin{corollary}\label{toricResult} 
Using Equation~(\ref{hatResult}) and the specialization $\boldY\mapsto (1,\ldots,1)$, the Poincar\'e series of the Ehrhart ring of the lattice simplex $\Delta_2^m$ is given by 
\[
P_K^{\E(\Delta_2^m)}(z)=\frac{(1+z)^{m+2}}{1-4z+z^2} \, .
\] 
\end{corollary}

\begin{remark}
Note that the structure of the Poincar\'e series in a single variable in this case does not fully represent the structure of the minimal resolution we construct.
Rather, there is cancellation after specialization.
This indicates that while rational single-variable Poincar\'e series can be useful for asymptotic approximation of Betti numbers, to inspire explicit construction of minimal resolutions, sometimes a more complex multivariate rational function is required.
\end{remark}

\begin{proof}[Proof of Theorem~\ref{mainTheorem}]
Our proof will proceed as follows: first, we establish the hypotheses required for inductively constructing our resolution; second, for the inductive step, we identify kernel elements; third, we prove those kernel elements generate the kernel; fourth, we prove that the resulting resolution is minimal; fifth, we show that this resolution results in a rational Poincar\'e series.

\textbf{Step 1: Establish inductive hypotheses for constructing the resolution.}
We will begin with the initial complex given in Example~\ref{complexExample}.
Using this as a base case, we will inductively construct a minimal free resolution $F$ of the type we desire.
To verify that the complex $F_{\leq 2}$ in Example~\ref{complexExample} is exact at $F_1$, assume that $f$ is an element in the kernel of $d_1$ with leading term supported on some $\delta_i$.
If $i$ is equal to $1$, $2$, or $3$, we may reduce $f$ by subtracting a monomial multiple of one of the elements $d_2(\epsilon_1), \ldots,d_2(\epsilon_{12})$ in a way that strictly reduces the leading term of $f$; this is possible since no element of the kernel of $d_1$ can have a unit as a leading coefficient.
By iterated reductions of this type, we produce an element in the kernel supported on only $\delta_4$.
By the definition of $\widehat{R}$ and $d_1$, such an element must be a linear combination of $d_2(\epsilon_{13})$, $d_2(\epsilon_{14})$, and $d_2(\epsilon_{15})$, and thus our complex is exact.

It is straightforward to verify that our base case given by $F_{\leq 2}$ in Example~\ref{complexExample} satisfies the following four hypotheses.
To state the hypotheses, suppose for the sake of induction that we have produced a complex $F_{\leq n}$ that is exact except at $F_0$ and $F_n$. 

\textbf{Hypothesis (Ordering):} Assume that for each $i$, $F_i$ is ordered with respect to $d_i$, and no element of the kernel of $d_n$ has leading coefficient equal to a unit.

\textbf{Hypothesis (Generator Poset):} For each generator $\epsilon$ of $F_i$, where $1\leq i\leq n$, each coefficient, up to $K$-scalar, of $d_i(\epsilon)-\LT(d_i(\epsilon))$ is either zero or lies strictly below $\LC(d_i(\epsilon))$ in the poset in Figure~\ref{fig:Aposet}.

\begin{figure}[ht]
\begin{tikzpicture}[scale=1]

\draw (-2,0) -- (-1,1)--(0,2)--(1,1)--(-2,0); 
\draw  (-2,1) to (-2,0); 
\draw (-2,0) -- (-6,1); 
\draw (-2,0) -- (-4,1); 

\draw[white,fill=white] (-2,0) circle [radius=0.25];
\node at (-2,0) {$x_1$};
%\draw[fill] (0,0) circle [radius=0.025];
%\node at (0,0) [below] {$x_1$};

\draw[white,fill=white] (0,2) circle [radius=0.25];
\node at (0,2) {$x_4$};
%\draw[fill] (0,2) circle [radius=0.025];
%\node at (0,2) [above] {$x_4$};
\draw[white,fill=white] (-1,1) circle [radius=0.25];
\node at (-1,1) {$x_2$};
%\draw[fill] (-1,1) circle [radius=0.025];
%\node at (-1,1) [right] {$x_2$};
\draw[white,fill=white] (1,1) circle [radius=0.25];
\node at (1,1) {$x_3$};
%\draw[fill] (1,1) circle [radius=0.025];
%\node at (1,1) [right] {$x_3$};
\draw[white,fill=white] (-2,1) circle [radius=0.375];
\node at (-2,1)  {$x_2x_4^{m-1}$};
%\draw[fill] (-2,1) circle [radius=0.025];
%\node at (-2,1) [above] {$x_2x_4$};
\draw[white,fill=white] (-4,1) circle [radius=0.275];
\node at (-4,1) {$x_3x_4^{m-1}$};
%\draw[fill] (-3,1) circle [radius=0.025];
%\node at (-3,1) [above] {$x_3x_4$};
\draw[white,fill=white] (-6,1) circle [radius=0.375];
\node at (-6,1) {$x_4^m$};
%\draw[fill] (-4,1) circle [radius=0.025];
%\node at (-4,1) [above] {$x_4^2$};
\end{tikzpicture}
\caption{}
\label{fig:Aposet}
\end{figure}

\textbf{Hypothesis (Cover Condition):} For each generator $\delta$ of a summand of $F_{i}$, where $i$ is at most $n-1$ and $n$ is at least two, let the leading coefficient $\LC(d_i(\delta))$ be the monomial $s$. Then the following holds:
\begin{itemize}
\item If $s\in\{x_1,x_2,x_3,x_2x_4^{m-1},x_3x_4^{m-1}\}$, then there are exactly four elements covering $\delta$ in $T_{\leq n}$, and their leading coefficients are $x_1$, $x_2$, $x_3$, and $x_4$. 
\item If $s=x_4$, then there are exactly three elements covering $\delta$ in $T_{\leq n}$, and their leading coefficients are $x_2x_4^{m-1}$, $x_3x_4^{m-1}$, and $x_4^m$. 
\item If $s=x_4^m$, then there are exactly three elements covering $\delta$ in $T_{\leq n}$, and their leading coefficients are $x_2$, $x_3$, and $x_4$. 
\end{itemize}
These are the only  values that $s$ takes.

\textbf{Hypothesis (Boundary Condition):} For each generator $\epsilon$ of a summand of $F_i$, where $i$ is at least one and $\LT(d_i(\epsilon))=s\delta$:  

If $s\in\{x_2x_4^{m-1},x_3x_4^{m-1},x_4^m\}$, then one of the following holds:
\begin{itemize}
\item[(i)] $d_i(\epsilon)=s\delta$
\item[(ii)] $d_i(\epsilon)=s\delta + \sigma x_1\delta'$, where $\sigma\in\{1,-1\}$, $\LC(d_{i-1}(\delta))=x_4$, and $\LC(d_{i-1}(\delta'))=x_4^m$
\end{itemize}
If $s\in\{x_2,x_3\}$, then one of the following holds:
\begin{itemize}
\item[(iii)] $d_i(\epsilon)=s\delta$ where $\LC(d_{i-1}(\delta))=t\;\text{and}\;st=0$
\item[(iv)] $d_i(\epsilon)=s\delta - x_1\delta'$, where {${\LT(d_{i-1}(\delta))=t\gamma}$} for $t\in\{x_2x_4^{m-1},x_3x_4^{m-1}\}$\\ with $st\neq0$ and ${\LT(d_{i-1}(\delta'))=x_4^m\gamma}$
\item[(v)] $d_i(\epsilon)=s\delta - x_1\delta'$ where $\LT(d_{i-1}(\delta))=t\gamma$ for $t\in\{x_1,x_2,x_3\}$ with $st\neq0$ and $\LT(d_{i-1}(\delta'))=x_4\gamma$
\end{itemize}

If $s=x_4$, then one of the following holds:
\begin{itemize}
\item[(vi)] $d_{i}(\epsilon)=s\delta$ where $\LC(d_{i-1}(\delta))\in\{x_2x_4^{m-1},x_3x_4^{m-1},x_4^m\}$
\item[(vii)] $d_{i}(\epsilon)=s\delta -\sigma x_1\delta'$ where $\sigma\in\{1,-1\}$, $d_{i-1}(\delta)=t\gamma-x_1\gamma'$ for $t\in\{x_2x_4^{m-1},x_3x_4^{m-1}\}$ and $\LT(d_{i-1}(\delta'))=x_4\gamma'$
\item[(viii)] $d_{i}(\epsilon)=s\delta - t\delta'+x_1\delta''$ where $d_{i-1}(\delta)=t\gamma-x_1\gamma'$ for $t\in\{x_1,x_2,x_3\}$, $\LT(d_{i-1}(\delta'))=x_4\gamma$, and $\LT(d_{i-1}(\delta''))=x_4\gamma'$
\item[(ix)] $d_i(\epsilon)=s\delta - t\delta'$ for $t\in\{x_1,x_2,x_3\}$ where $\LT(d_{i-1}(\delta))=t\gamma$ and $\LT(d_{i-1}(\delta'))=x_4\gamma$
\end{itemize}
If $s=x_1$, then \begin{itemize}\item [(x)]$d_i(\epsilon)=s\delta$\end{itemize}

\textbf{Step 2: Inductive construction of kernel elements.}
Assume that hypotheses (Ordering), (Generator Poset), (Cover Condition), and (Boundary Condition) are satisfied by our complex $F_{\leq n}$, exact except at $F_0$ and $F_n$. 
We will now use hypotheses (Ordering), (Generator Poset), (Cover Condition), and (Boundary Condition) to show that for each generator $\epsilon$ of $F_n$, there exists a set of homogeneous kernel elements whose leading term is supported on $\epsilon$ and whose leading coefficients satisfy hypothesis (Cover Condition). 

Specifically, assume that $\epsilon$ is such that $\LC(d_n(\epsilon))=s$:
\begin{itemize}
\item For each of the cases $s\in\{x_1,x_2,x_3,x_2x_4^{m-1},x_3x_4^{m-1}\}$ we find an element $f_i$ of $\ker d_n$ with leading term $u\epsilon$ for $u\in\{x_1,x_2,x_3,x_4\}$. 
\item For $s=x_4$ we find a kernel element $f_i$ with leading term $u\epsilon$ for each \\
$u\in\{x_2x_4^{m-1},x_3x_4^{m-1},x_4^m\}$. 
\item For $s=x_4^m$ we find a kernel element $f_i$ with leading term $u\epsilon$ for each $u\in\{x_2,x_3,x_4\}$. 
\end{itemize}
Observe that this is precisely what is needed to extend our resolution while satisfying (Cover Condition).
We will denote this collection of kernel elements by $\{f_i\}$ and let them be ordered by $\prec$ on the minimal supports (with tie breaking by $\prec_{\widehat{R}}$ on the leading coefficients).

Let $\epsilon$ be a generator of $F_n$ with $\LT(d_n(\epsilon))=s\delta$ and $\LT(d_{n-1}(\delta))=t\gamma$. 
We construct the elements set $\{f_i\}$ in a case-by-case manner as follows.

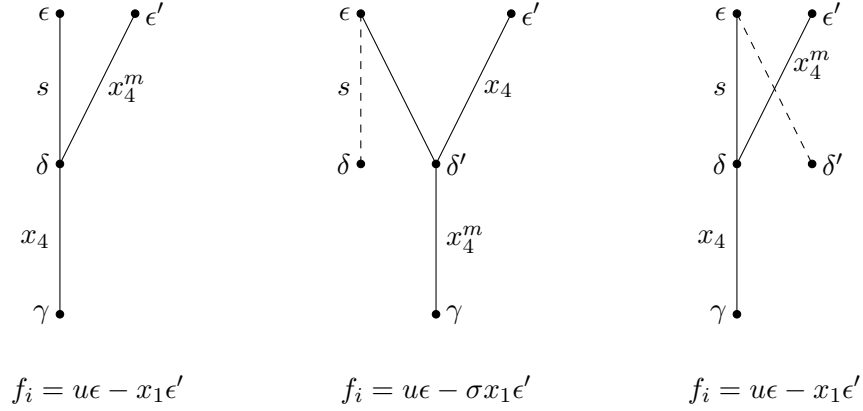
\begin{figure}
\begin{tikzpicture}[scale=1]
\draw (0,0) -- (0,2)--(0,4);
\draw (1,4)--(0,2);
\draw[fill] (0,0) circle [radius=0.05];
\draw[fill] (1,4) circle [radius=0.05];
\draw[fill] (0,2) circle [radius=0.05];
\draw[fill] (0,4) circle [radius=0.05];
\node at (0,2) [left]{$\delta$};
\node at (0,4) [left]{$\epsilon$};
\node at (1,4) [right]{$\epsilon'$};
\node at (0,0) [left]{$\gamma$};
%\node at (0,-1/2) {$u$};
%\node at (1,-1/2) {$-x_1$};
\node at (0,3) [left]{$s$};
\node at (1/2,3) [right]{$x_4^m$};
\node at (0,1) [left]{$x_4$};
\node at (0+1/2,-1) {$f_i=u\epsilon-x_1\epsilon'$};

\draw [dashed](4,4) -- (4,2);
\draw (4,4) -- (5,2);
\draw (5,0) -- (5,2);
\draw (5,2) -- (6,4);
\draw[fill] (4,4) circle [radius=0.05];
\draw[fill] (6,4) circle [radius=0.05];
\draw[fill] (4,2) circle [radius=0.05];
\draw[fill] (5,2) circle [radius=0.05];
\draw[fill] (5,0) circle [radius=0.05];
\node at (4,2) [left]{$\delta$};
\node at (5,2) [right]{$\delta'$};
\node at (5,1) [right]{$x_4^m$};
\node at (5,0) [right]{$\gamma$};
\node at (4,4) [left]{$\epsilon$};
\node at (6,4) [right]{$\epsilon'$};
%\node at (4,-1/2) {$u$};
\node at (4,3) [left]{$s$};
%\node at (4+3/4,1/2) {$\sigma x_1$};
\node at (5+1/2,3) [right]{$x_4$};
%\node at (6,-1/2) {$-\sigma x_1$};
\node at (5,-1) {$f_i=u\epsilon-\sigma x_1\epsilon'$};

\draw (9,0) -- (9,2)--(9,4);
\draw (10,4)--(9,2);
\draw [dashed](9,4) -- (10,2);
\draw[fill] (9,0) circle [radius=0.05];
\draw[fill] (10,4) circle [radius=0.05];
\draw[fill] (9,2) circle [radius=0.05];
\draw[fill] (9,4) circle [radius=0.05];
\draw[fill] (10,2) circle [radius=0.05];
\node at (9,2) [left]{$\delta$};
\node at (10,2) [right]{$\delta'$};
\node at (9,4) [left]{$\epsilon$};
\node at (10,4) [right]{$\epsilon'$};
\node at (9,0) [left]{$\gamma$};
%\node at (9,-1/2) {$u$};
%\node at (10,-1/2) {$-x_1$};
\node at (9,3) [left]{$s$};
\node at (10,3) [above]{$x_4^m$};
\node at (9,1) [left]{$x_4$};
\node at (9+1/2,-1) {$f_i=u\epsilon-x_1\epsilon'$};
\end{tikzpicture}
\caption{The case $s\in \{x_2x_4^{m-1}, x_3x_4^{m-1}\}$.}
\label{fig:case1}
\end{figure}

\textbf{Case: $s\in \{x_2x_4^{m-1}, x_3x_4^{m-1}\}$.} 
We refer to Figure~\ref{fig:case1} throughout this argument.
By (Boundary Condition), $d_n(\epsilon)$ is either equal to $s\delta$ or $s\delta+\sigma x_1\delta'$, where $\LC(d_{n-1}(\delta'))=x_4^m$. 
We suppose $u\in\{x_1,x_2,x_3,x_4\}$ and consider three subcases.

\begin{itemize}
\item If $d_n(u\epsilon)=us\delta$ is zero, we set $f_i:= u\epsilon$.
This can only happen in two situations, either when $s=x_3x_4^{m-1}$ and $u$ is equal to $x_1$, $x_3$, or $x_4$, or else when $s=x_2x_4^{m-1}$ and $u$ is equal to $x_1$, $x_2$, or $x_4$.
Note that this assignment of $f_i$ satisfies the four inductive hypotheses.
\item If $d_n(\epsilon)=s\delta$ and $d_n(u\epsilon)\neq 0$, then $d_n(u\epsilon)$ must be equal to $x_1x_4^m\delta$.
This can only happen when either $s=x_2x_4^{m-1}$ and $u=x_3$ or when $s=x_3x_4^{m-1}$ and $u=x_2$, and both situations require consideration of the binomial relation $x_2x_3=x_1x_4$.
See the left-hand skematic in Figure~\ref{fig:case1} to illustrate the following argument.
Since $\delta\in F_{n-1}$ is covered in $T_{\leq n}$ by an element $\epsilon\in F_n$ having $\LC(d_n(\epsilon))$ equal to one of $x_2x_4^{m-1}$ or $x_3x_4^{m-1}$, by (Cover Condition) we have that $\LC(d_{n-1}(\delta))=x_4$.
Thus, again by (Cover Condition), there exists $\epsilon'\in F_n$ with $\LT(d_n(\epsilon'))=x_4^m\delta$. 
It follows that  $d_n(x_1\epsilon')=x_1x_4^m\delta$ since by (Generator Poset), each coefficient of $d_n(\epsilon')-\LT(d_n(\epsilon'))$ is either zero or $x_1$, and $x_1^2=0$.
Thus, $d_n(u\epsilon-x_1\epsilon')=0$ and we can set $f_i:=u\epsilon-x_1\epsilon'$.
Observe that this assignment of $f_i$ satisfies the four inductive hypotheses.
\item If $d_n(\epsilon)=s\delta +\sigma x_1\delta'$, observe that (Boundary Condition) implies that since $d_n(\epsilon)=s\delta +\sigma x_1\delta'$, we have $\LC(d_{n-1}(\delta'))=x_4^m$.
We have three subsubcases that arise in this subcase, and note the assignment of $f_i$ given in each of them satisfies the four inductive hypotheses.
\begin{itemize}
\item If $d_n(u\epsilon)$ is equal to zero, we set $f_i:= u\epsilon$.
This will happen when $u=x_1$ and for certain pairs of $u$ and $s$ when $u=x_2$ or $u=x_3$, with the remaining pairs handled in the subsubcase $su=x_1x_4^m$ below.
\item If $u=x_4$, then $us=0$ for both possible values of $s$.
Thus, $d_n(u\epsilon)=\sigma x_1x_4\delta'$.
See the center skematic in Figure~\ref{fig:case1} illustrating the following argument.
Since $\LC(d_{n-1}(\delta'))=x_4^m$, by (Cover Condition) there exists $\epsilon'$ such that $\LC(d_n(\epsilon'))=x_4\delta'$.
Since by (Boundary Condition) any coefficient of $d_n(\epsilon')-\LT(d_n(\epsilon'))$ is either zero, $x_1$, $x_2$, or $x_3$, and multiplying any of these variables by $x_1$ results in a zero, we have that $d_n(u\epsilon-\sigma x_1\epsilon')=0$.
Thus, we can set $f_i:=u\epsilon-\sigma x_1\epsilon'$.
\item As before, $su=x_1x_4^m$ can only happen in two situations, when either $s=x_2x_4^{m-1}$ and $u=x_3$ or when $s=x_3x_4^{m-1}$ and $u=x_2$.
See the right-hand skematic in Figure~\ref{fig:case1} illustrating the following argument.
In either event, we have that $d_n(u\epsilon)=x_1x_4^m\delta$, since multiplying $x_2$ or $x_3$ by the $x_1$ in the coefficient of $\delta'$ will zero out that term. 
As in a previous case, by (Cover Condition), we can find an $\epsilon'$ such that $\LT(d_n(\epsilon'))=x_4^m\delta$ and $d_n(x_1\epsilon')=x_1x_4^m\delta$.
In this case, $d_n(u\epsilon-x_1\epsilon')=0$, hence we set $f_i:=u\epsilon-x_1\epsilon'$.
\end{itemize}
\end{itemize}

\textbf{Case $s=x_4^m$:}
By (Boundary Condition), $d_n(\epsilon)$ is equal to $s\delta$ or $s\delta+\sigma x_1\delta'$, where $\LC(d_{n-1}(\delta'))=x_4^m$. 

\begin{itemize} \item If $u$ equals $x_2$ or $x_3$, then since $us$ and $ux_1$ are both equal to zero, $d_n(u\epsilon)=0$ and we set $f_i:=u\epsilon$.
\item If instead $u$ is equal to $x_4$, then $d_n(u\epsilon)$ is either zero or is equal to $\sigma x_1x_4\delta'$ for some $\delta'$ with $\LC(d_{n-1}\delta')=x_4^m$. By  (Cover Condition), since $\LC(d_{n-1}\delta')=x_4^m$, there exists a generator $\epsilon'$ of $F_n$ with $\LT(d_n(\epsilon'))=x_4\delta'$. By (Boundary Condition) applied to $\epsilon'$, since $\LC(d_{n-1}(\delta'))=x_4^m$, we have the stronger condition that $d_n(\epsilon')=x_4\delta'$, so that $d_n(u\epsilon-\sigma x_1\epsilon')=0$. Thus we set $f_i=u\epsilon-\sigma x_1\epsilon'$.

\item If $u=x_1$, then $\LT(d_n(u\epsilon))=x_1x_4^m\delta$, which is nonzero. By construction, any generator $\epsilon'$ which is after $\epsilon$ in the ordering $\prec$ of generators of $F_n$ has   the leading term of its image under $d_n$ supported on generators that are after or equal to $\delta$ in the ordering of generators of $d_{n-1}$.  By (Cover Condition), $\epsilon$ is the $\prec$-maximal generator whose image under $d_n$ is supported on $\delta$. Consequently, the equation 
\[
d_n(u\epsilon)=-d_n\left(\sum_{\epsilon\prec\epsilon_k} r_k\epsilon_k\right)
\] 
has no solutions for $r_k$ in $\widehat{R}$, and so no homogeneous kernel element of $d_n$ has leading term $x_1\epsilon$. Since any monomial of $\widehat{R}$ of degree greater than one is divisible by $x_4$, no minimal generator of $\ker d_n$ has leading term $u\epsilon$ where $u$ is different from $x_2$, $x_3$, or $x_4$.
\end{itemize}

\textbf{Case $s\in\{x_1,x_2,x_3\}$:}
By (Boundary Condition), $d_n(\epsilon)=s\delta$ or $d_n(\epsilon)=s\delta-x_1\delta'$.
\begin{itemize} \item Let $u\in\{x_1,x_2,x_3\}$. Observe that $us$ is either zero or equal to $x_2x_3$, and that $ux_1$ is always zero. Thus $d_n(u\epsilon)$ is equal to $us\delta$, and is either zero or equal to $x_1x_4\delta$ (under the equivalence $x_2x_3=x_1x_4$). In the first case we set $f_i=u\epsilon$. In the second case, we note that since by hypothesis $s\delta:=\LT(d_n(\epsilon))$ is among $x_1\delta$, $x_2\delta$, and $x_3\delta$, by (Cover Condition) there exists $\epsilon'$ such that $\LT(d_n(\epsilon'))=x_4\delta$. By (Generator Poset), $d_n(x_1\epsilon')=x_1x_4\delta$ since $x_1^2=x_1x_2=x_1x_3=0$, so that $d_n(u\epsilon-x_1\epsilon')=0$. We set $f_i=u\epsilon-x_1\epsilon'$.

\item Let instead $u=x_4$, and consider the two cases: $d_n(\epsilon)=s\delta$ and $d_n(\epsilon)=s\delta-x_1\delta'$. Again note that by (Cover Condition) there exists $\epsilon'$ with $\LT(d_n(\epsilon'))=x_4\delta$ and recall that by (Poset Condition), $d_n(x_1\epsilon')=x_1x_4\delta$. 
\begin{itemize}
\item In the case that $d_n(\epsilon)=s\delta=x_1\delta$, we have that $d_n(u\epsilon-x_1\epsilon')=0$, and we set $f_i=u\epsilon-x_1\epsilon'$. Let instead $d_n(\epsilon)=s\delta$ where $s$  equals $x_2$ or $x_3$, and note that by (Boundary Condition), $\LC(d_{n-1}(\delta))=t$ where $st=0$. Then by (Boundary Condition) there are two possibilities for $d_n(s\epsilon')$. Because $sx_1=0$, either $d_n(s\epsilon')$ is equal to $x_4s\delta$, or else it is equal to $x_4s\delta-st\delta'$ for some generator $\delta'$ of $F_{n-1}$. As we have established, $st=0$, and so $d_n(s\epsilon')=x_4s\delta$, so that $d_n(u\epsilon-s\epsilon')=0$. We set $f_i=u\epsilon-s\epsilon'$. 

\item If $d_n(\epsilon)=s\delta-x_1\delta'$, then by (Boundary Condition), $s$ is equal to $x_2$ or $x_3$. There are two possibilities, both having $\LT(d_{n-1}(\delta))=t\gamma$ where $st\neq0$. {\bf{The first possibility}} is that $t$ is among $\{x_2x_4^{m-1},x_3x_4^{m-1}\}$ and $\LT(d_{n-1}(\delta'))=x_4^m\gamma$. In this case, by (Cover Condition) there exists $\epsilon''$ such that $\LT(d_n(\epsilon'')=x_4\delta'$. By (Boundary Condition) applied to $\epsilon'$, we see that since $t$ is among $\{x_2x_4^{m-1},x_3x_4^{m-1}\}$, $d_n(\epsilon')$ is either $x_4\delta$ or $x_4\delta-\sigma x_1\delta'$. In either case, $d_n(s\epsilon')$ is equal to $x_4s\delta$. Also by (Boundary Condition), since $\LC(d_{n-1}(\delta'))$ is $x_4^{m-1}$, we see that $d_n(\epsilon'')$ is equal to $x_4\delta'$. Thus we see that $d_n(u\epsilon-s\epsilon'+x_1\epsilon'')=0$ and we set $f_i=u\epsilon-s\epsilon'+x_1\epsilon''$.

{\bf{The second possibility}} is that $t$ is among $\{x_1,x_2,x_3\}$ and $\LT(d_{n-1}(\delta'))=x_4\gamma$. In this case we see that by (Boundary Condition), since $sx_1=0$, we have that $d_n(s\epsilon')$ is equal to $x_4s\delta-st\delta'$. Observe that for $s$ equal to $x_2$ or $x_3$ and $t$ among $\{x_1,x_2,x_3\}$, the fact that $st$ is nonzero implies that $st$ equals $x_2x_3=x_1x_4$. It follows that $d_n(u\epsilon-s\epsilon')=0$ and we set $f_i=u\epsilon-s\epsilon'$. 
\end{itemize}
\end{itemize}

\textbf{Case $s=x_4$:}

\begin{itemize}
\item For $u$ any monomial of $\N$--degree less than or equal to $m-1$, or for $u=x_1x_4^{m-1}$, $d_n(u\epsilon)$ is nonzero, and as argued earlier, since the leading coefficient of $\epsilon$ is $x_4$, (Cover Condition) implies that the ordering $\prec$ of the generators of $F_n$ precludes any homogeneous kernel element of $d_n$ having leading term $u\epsilon$. 

The monomials having $\N$--degree equal to $m$ are $x_2x_4^{m-1}$, $x_3x_4^{m-1}$, and $x_4^{m}$. We show that each of these monomials is the leading coefficient of homogeneous kernel element of $d_n$ having leading support $\epsilon$.

If $u\in\{x_2x_4^{m-1},x_3x_4^{m-1}\}$, then by (Boundary Condition), either $d(u\epsilon))=0$ or $d(u\epsilon)=-x_1x_4^m\delta'$, where $\LC(d(\delta'))=x_4$. In the first case we set $f_i=u\epsilon$. In the latter case, by (Cover Condition) there exists $\epsilon'$ such that $\LT(d(\epsilon'))=x_4^m\delta'$. By (Generator Poset), $d_n(x_1\epsilon')=x_1x_4^m\delta'$, so that $d_n(u\epsilon+x_1\epsilon')=0$ and we set $f_i=u\epsilon+x_1\epsilon'$.

If $u=x_4^m$, then by (Boundary Condition) we have that either $d_n(\epsilon)=x_4\delta$, so that $d_n(u\epsilon)=0$,
  or $d_n(u\epsilon)=\sigma x_1x_4^m\delta'$, where $\LC(d_{n-1}(\delta'))=x_4$. In the latter case, we have by (Cover Condition) that there exists $\epsilon'$ such that $\LT(d_n(\epsilon'))=x_4^m\delta'$. It follows that $d_n(u\epsilon-\sigma x_1\epsilon')=0$ and we set $f_i=u\epsilon-\sigma x_1\epsilon'$. 
\end{itemize}

Thus, we have constructed a set of kernel elements $\{f_i\}$ satisfying the properties stated at the beginning of this step.

\textbf{Step 3: Proof that the elements $\{f_i\}$ generate the kernel.}
Note that by our inductive construction, no element of the kernel has a leading coefficient equal to a unit in $\widehat{R}$.
Given a homogeneous element $f$ of $\ker d_n$ with $\LT(f)=v\epsilon$, we have by hypothesis (Cover Condition) that $\LC(d_n(\epsilon))=s$ is among the cases above. If $s\in\{x_1,x_2,x_3,x_2x_4^{m-1},x_3x_4^{m-1}\}$, then since by hypothesis the leading coefficient of $f$ is divisible by some variable, $f$ may be reduced by one of the $f_i$'s with the same minimal support to a kernel element with strictly larger leading term. 

If $s=x_4$, then as established above, $v$ is divisible by one of $x_2x_4^{m-1}$, $x_3x_4^{m-1}$, or $x_4^m$, and  $f$ may be reduced by one of the $f_i$'s with the same minimal support to a kernel element with strictly larger leading term.

As there is a finite collection of possible leading terms, this process will reduce $f$ to zero, showing that the $f_i$ are a generating set for $\ker d_n$.

\textbf{Step 4: Proof that this is a minimal resolution.}

Now that we have proved that the set of $f_i$'s generate the kernel of $d_n$, we proceed by augmenting $F_{\leq n}$ with a free $\widehat{R}$-module with basis in bijection with the $f_i$'s.
Define a map $d_{n+1}$ sending each new basis element to its associated $f_i$, and the result is a complex $F_{\leq n+1}$, which is exact except at $F_0$ and $F_{n+1}$.
We need to show that our choice of kernel generators $f_i$ is a minimal set of generators.

Since none of the $f_i$ whose leading coefficient is a variable can be written as an $\widehat{R}$-linear combination of the others, we only need consider the minimality of $f_i$'s whose leading coefficient is $x_kx_4^{m-1}$, where $k$ is 2, 3, or 4.
Let $\LT(f_i)=x_kx_4^{m-1}\epsilon$, where $\LT(d_n(\epsilon))=x_4\gamma$ and $f_j\neq f_i$ be such that $[\epsilon](s_jf_j)$, the coefficient of $s_jf_j$ on the summand generated by $\epsilon$, is also $x_kx_4^{m-1}$. 
Since, by construction, the only coefficients appearing in $f_j$ are in $\{x_1,x_2,x_3,x_4,x_2x_4^{m-1},x_3x_4^{m-1},x_4^{m}\}$, we see by divisibility that $[\epsilon](f_j)$ is $x_2$, $x_3$, or $x_4$. 
Also by construction, $[\epsilon](f_j)=x_4$ implies that $\LT(f_j)=x_4\epsilon$, which is a contradiction with $\LT(d_n(\epsilon))=x_4\gamma$, thus $[\epsilon](f_j)$ is $x_2$ or $x_3$, and $s_j=x_4^{m-1}$. 
Again by construction, $\LT(f_j)=x_4\epsilon'$, where $\LT(d_n(\epsilon'))=x_k\gamma$, hence $\LC(s_jf_j)=x_4^m\neq0$, so that if $\displaystyle\LT\left(\sum_{\ell\neq j}s_\ell f_\ell\right)=x_kx_4^{m-1}\epsilon$, then there exists $\ell$ such that $[\epsilon']s_\ell f_\ell=x_4^m$. 
But no $f_\ell$ with $\ell\neq j$ can have $[\epsilon']f_\ell$ divisible by $x_4$. 
Thus $f_i$ is minimal, and the collection of $f_i$'s is in fact a minimal generating set. 

\textbf{Step 5: Verification that the inductive result satisfies the four hypotheses and produces the matrix $A$, computation of the rational Poincar\'e series.}
It is immediate from the construction above that the hypotheses (Ordering), (Generator Poset), (Cover Condition), and (Boundary Condition) are satisfied by the augmented complex $F_{\leq n+1}$.
Further, it is immediate from these four hypotheses that the matrix $A$ given in the statement of the theorem is correct.
We therefore have established an inductive construction of the minimal free resolution of $K$ over $\widehat{R}$.

Having constructed our desired minimal free resolution, by Lemma~\ref{matrixLemma} we have that $P^{\widehat{R}}_K(z,t)$ is rational of the form 
\[
\frac{f(z,\boldY)}{\chi(z,\boldY,1)} \, 
\] 
where $\chi(z,\boldY,t)$ is the characteristic polynomial of the matrix $A$.
Computation in Macaulay2 gives that 
\[
\chi(z,\boldY,1)=1-z(\boldY^{\deg(x_1)}+\boldY^{\deg(x_2)}+\boldY^{\deg(x_3)})-z^2(\boldY^{\deg(x_2x_4^{m})}+\boldY^{\deg(x_3x_4^m)}+\boldY^{\deg(x_4^{m+1})})+z^3\boldY^{\deg(x_1x_4^{m+1})}.
\]
Using the minimal resolution construction given above, we compute that  $F_{\leq3}(z,\boldY)\cdot \chi(z,\boldY,1)$ is given by
\[
1+z\boldY^{\deg(x_4)}+z^4E(z,\boldY) \, ,
\] 
where $E(z,\boldY)$ is a polynomial in $K[z,\boldY]$. 
By Lemma~\ref{matrixLemma}, 
\[
f(z,\boldY)=\chi(z,\boldY,1)\cdot F(z,\boldY) \, ,
\] so that 
\[
f(z,\boldY)-F_{\leq3}(z,\boldY)\cdot \chi(z,\boldY,1)=(F(z,\boldY)-F_{\leq3}(z,\boldY))\cdot \chi(z,\boldY,1)
\] 
is a polynomial divisible by $z^4$. 
Since the $z$-degree of $f(z,\boldY)$ is at most three (by the degree of $\chi(z,\boldY,1)$), we see that $f(z,\boldY)=1+z\boldY^{\deg(x_4)}$, and the rational form follows.
\end{proof}

\bibliographystyle{amsplain}
\bibliography{BrianBib}
\end{document}